\documentclass[a4paper, oneside, reqno, 12pt]{amsart}
\usepackage[utf8]{inputenc}
\usepackage{amsmath, amssymb, amsfonts}
\usepackage{fullpage}
\usepackage{textcomp, cmap, comment}
\usepackage{mathtools}
\usepackage[nobysame, initials]{amsrefs}
\usepackage[english]{babel}

\usepackage{comment}
\usepackage{color, xcolor}
\usepackage{graphicx}
\usepackage{pgf}
\usepackage{pgfplots}
\usepackage{relsize}

\usepackage[justification=centering]{caption}
\usepackage{subcaption}
\usepackage{wrapfig}

\theoremstyle{plain}
\newtheorem{theorem}{Theorem}
\newtheorem{conjecture}[theorem]{Conjecture}

\newtheorem{proposition}[theorem]{Proposition}

\newtheorem{lemma}[theorem]{Lemma}

\theoremstyle{definition}
\newtheorem{problem}[theorem]{Problem}

\theoremstyle{remark}

\DeclareMathOperator{\cost}{cost}

\DeclareMathOperator{\conv}{conv}

\usepackage{mathrsfs}

\usepackage{hyperref}
\hypersetup{
    colorlinks   = true, 
    urlcolor     = blue, 
    linkcolor    = red, 
    citecolor   = green
}

\title{Intersecting diametral balls induced\\ by a geometric graph II}
\author[P.~Barabanshchikova, A.~Polyanskii]{{Polina Barabanshchikova and Alexander~Polyanskii}}

\address{Polina Barabanshchikova,
\newline\hphantom{iii} Moscow Institute of Physics and Technology, Institutskiy per. 9, Dolgoprudny, Russia 141700
}
\email{\href{mailto:barabanshchikova.piu@phystech.edu}{barabanshchikova.piu@phystech.edu}}

\address{Alexander Polyanskii,
\newline\hphantom{iii} Institute of Mathematics and Informatics, Bulgarian Academy of Sciences, Bulgaria, Sofia 1113, Acad. G. Bonchev Str., Bl. 8
}
\email{\href{mailto:alexander.polyanskii@gmail.com}{alexander.polyanskii@gmail.com}}
\urladdr{\url{http://polyanskii.com}}

\keywords{Infinite descent, convex optimization, Tverberg theorem, max-sum tree, max-sum matching, alternating cycle}
\subjclass[2010]{51K99, 05C50, 51F99, 52C99, 05A99}

\begin{document}

\thispagestyle{empty}

\begin{abstract}
For a graph whose vertices are points in $\mathbb R^d$, consider the closed balls with diameters induced by its edges. The graph is called a \textit{Tverberg graph} if these closed balls share a common point. 

A \textit{max-sum tree} of a finite point set $X \subset \mathbb R^d$ is a tree with vertex set $X$ that maximizes the sum of Euclidean distances of its edges among all trees with vertex set $X$. Similarly, a \textit{max-sum matching} of an even set $X \subset \mathbb R^d$ is a perfect matching of $X$ maximizing the sum of Euclidean distances between the matched points among all perfect matchings of~$X$.

We prove that a max-sum tree of any finite point set in $\mathbb R^d$ is a Tverberg graph, which generalizes a recent result of Abu-Affash et al., who established this claim in the plane. Additionally, we provide a new proof of a theorem by Bereg et al., which states that a max-sum matching of any even point set in the plane is a Tverberg graph. Moreover, we proved a slightly stronger version of this theorem.
\end{abstract}

\maketitle

\section{Introduction}

In 1966, Helge Tverberg~\cite{tverberg1966generalization} proved that for any $(r-1)(d+1)+1$ points in $\mathbb{R}^d$, there exists a partition of them into $r$ parts whose convex hulls intersect. This paper studies a variation of Tverberg's theorem recently introduced by Huemer et al.~\cite{huemer2019matching} and Sober{\'o}n and Tang~\cite{soberon2021tverberg}. We further develop the approach presented in the works of Pirahmad et al.~\cite{pirahmad2022intersecting} and the authors~\cite{barabanshchikova2022intersecting}. Specifically, we use Proposition~\ref{proposition: main tool}, a key tool from~\cite{barabanshchikova2022intersecting} that was implicitly used in~\cite{pirahmad2022intersecting}. For a discussion of this proposition, we refer the reader to Section~\ref{section: preliminaries}.

For any graph in this paper, we assume that its vertex set is a finite subset of~$\mathbb{R}^d$. The \textit{cost} of a graph~$G$, denoted by $\cost G$, is the sum of Euclidean distances between the pairs of vertices connected by an edge in~$G$. We define a \textit{max-sum tree} of a finite point set $X\subset \mathbb{R}^d$ as a tree with vertex set~$X$ that maximizes the cost among all trees with vertex set~$X$. Similarly, we define a \textit{max-sum matching} of an even point set $X\subset \mathbb{R}^d$ as a perfect matching with vertex set~$X$ that maximizes the cost among all perfect matchings with vertex set~$X$. It will be convenient to slightly abuse notation and sometimes consider a matching as a set of edges, not as a graph.

For two points $x,y\in \mathbb R^d$, we denote by $B(xy)$ the closed Euclidean ball with diameter~$xy$. We say that the ball $B(xy)$ is \textit{induced} by $xy$. A graph $G$ is called a \textit{Tverberg graph} if
\[
    \bigcap_{xy\in E(G)} B(xy) \neq \emptyset.
\]
Similarly, a graph is an \textit{open Tverberg
graph} if the open balls with diameters induced by its edges intersect.

Very recently, Abu-Affash et al.~\cite{abuaffash2022} proved that a max-sum tree of any finite point set in the plane is a Tverberg graph. In this paper, we generalize their result to higher dimensions.
\begin{theorem}
\label{theorem tree}
A max-sum tree of any finite point set in $\mathbb R^d$ is a Tverberg graph.
\end{theorem}

In 2021, Bereg et al.~\cite{bereg2023maximum}*{Theorem~3.14} showed that a max-sum matching of any even set in the plane is a Tverberg graph. Our second result is a new proof of a slightly stronger version of their result under the assumption that points are distinct; see the discussion in Section~\ref{section discussion}.
%Moreover, they gave an example of a set such that the intersection of the discs induced by its max-sum matching is a singleton; see \cite{bereg2023maximum}*{Paragraphs after Theorem~3.2}. Specifically, consider six points in the plane, where three points are the vertices of a triangle, and the other three points coincide with any interior point of this triangle.
\begin{theorem}
\label{theorem matching}
A max-sum matching of any even set of distinct points in the plane is an open Tverberg graph.
\end{theorem}

Remark that the condition that points are distinct is crucial. Indeed, Bereg et al. demonstrated that the intersection of the closed discs induced by a max-sum matching can be a singleton, that is, this matching is not an open Tverberg graph; see the third paragraph after Theorem~3.2 in~\cite{bereg2023maximum}. For example, consider 4 points $a, b, c$, and $d$ such that the points $b$ and $c$ coincide and lie in the interior of a line segment $ad$. One of its max-sum matchings\footnote[1]{One can easily verify that the other max-sum matchings $\{ac,bd\}$ and $\{ad,cb\}$ also satisfy this property.} consists of two edges $ab$ and $cd$, and the corresponding closed discs $B(ab)$ and $B(cd)$ have a single point in common. However, the open discs induced by these edges have an empty intersection.
\smallskip

The paper is organized as follows. In Section~\ref{section: preliminaries}, we state and discuss Proposition~\ref{proposition: main tool}, the main ingredient of the proofs. In Sections~\ref{section proof of theorem tree} and~\ref{section proof of theorem matching}, we prove Theorems~\ref{theorem tree} and~\ref{theorem matching}, respectively. In Section~\ref{section discussion}, we analyze the connection between Theorem~\ref{theorem matching} and the result of Bereg et al.~\cite{bereg2023maximum}. We conclude with a discussion of open problems in Section~\ref{section: open problems}.

\section{Preliminaries}
\label{section: preliminaries}
Recall standard definitions from convex geometry.
A~set $X \subseteq \mathbb R^d$ is called \textit{convex} if for any two points $a,b\in X$, the line segment $ab$ also lies in~$X$.
The \textit{convex hull} of a set~$X \subset \mathbb R^d$, denoted by $\conv X$, is the intersection of all convex sets containing~$X$. There is a standard alternative way to define $\conv X$ using convex combinations. Recall that a \textit{convex combination} of a finite point set in $\mathbb R^d$ is a linear combination of these points with non-negative coefficients that sum up to $1$. Then the convex hull $\conv X$ is the set of all convex combinations of each finite subset of $X$.

A function $f:\mathbb R^d \rightarrow \mathbb R$ is called \textit{convex} if for all $0\leq \lambda \leq 1$ and all points $x, y\in \mathbb R^d$ it satisfies
\[
    f(\lambda x+(1-\lambda)y)\leq \lambda f(x)+(1-\lambda)f(y).
\]

The proofs of Theorems~\ref{theorem tree} and~\ref{theorem matching} both rely on the following simple observation about convex functions.

\begin{proposition}[\cite{barabanshchikova2022intersecting}*{Proposition~10}]
\label{proposition: main tool}
For convex functions $f_1, \dots , f_m : \mathbb R^d \to \mathbb R$, let the function $f : \mathbb R^d \to \mathbb R$ be defined by $f(x) \coloneqq
\max \{f_1(x), \dots , f_m(x)\}$. For any point $x \in \mathbb R^d$, put $I(x) \coloneqq \{1 \leq i \leq m : f_i(x) =
f(x)\}$.

    If the function $f$ attains its global minimum at a point $y \in \mathbb R^d$ and, for each $i\in I(y)$, the function $f_i$ is differentiable at $y$, then we have 
    \[
        o\in \conv \big\{ \nabla f_i (y): i\in I(y)\big\},
    \]
    where $o$ is the origin of $\mathbb R^d$ and $\nabla f_i (y)$ stands for the gradient of the function $f_i$ at the point $y = (y_1, \dots, y_d) \in \mathbb R^d$, that is,
    \[
        \nabla f_i (y) = \Big(\frac{\partial f_i}{\partial y_1} (y), \dots,  \frac{\partial f_i}{\partial y_d} (y)\Big).
    \]
\end{proposition}

%This proposition can be applied in the context of Tverberg graphs as outlined below. Let $\mathcal G$ denote a collection of graphs (matchings, Hamiltonian cycles, trees, etc.) with a shared finite vertex set $X\subset \mathbb R^d$. Our goal is to demonstrate that if a properly chosen function $Q:\mathcal G\to \mathbb R$ attains its maximum at a graph $G\in \mathcal G$, then $G$ is a Tverberg graph. In this paper, the cost function serves as $Q$.
This proposition can be applied in the context of Tverberg graphs as outlined below. Let $\mathcal G$ be a collection of graphs (matchings, Hamiltonian cycles, trees, etc.) with a common finite vertex set $X\subset \mathbb R^d$. We aim to show that if a properly chosen function $Q:\mathcal G\to \mathbb R$ attains its maximum at a graph $G\in \mathcal G$, then $G$ is a Tverberg graph. In this paper, the sum of Euclidean distances between vertices connected by an edge serves as $Q$.

Next, we introduce an auxiliary function $f_{G}: \mathbb R^d\to \mathbb R$ defined by
\[
    f_G(x)\coloneqq\max \{ f_{ab}(x): ab\in E(G) \},
\]
where differential functions $f_{ab}:\mathbb R^d\to \mathbb R$ are chosen in a way that $f_{ab}(x)\leq 0$ if and only if $x\in B(ab)$. Assume that the origin $o$ is a global minimum point of $f_G$, so if $f_G(o)\leq 0$, the desired claim is proved. Hence, we suppose $f_G(o)>0$. By Proposition~\ref{proposition: main tool} applied to the functions $f_{ab}$, the following balancing condition holds
\[
o\in \conv\big\{\nabla f_{ab}(o): ab\in E(G) \text{ with }f_{ab}(o)=f_G(o)\big\}.
\]
Finally, by analyzing this condition and using $f_G(o)>0$, we try to find a graph $G' \in \mathcal G$ with $Q(G') > Q(G)$. If we succeed in that, we obtain a contradiction to the maximality of $G$, and hence, $f_G(o)\leq 0$, which finishes the proof.

For example, Pirahmad et al.~\cite{pirahmad2022intersecting} use this approach. They show that for any $n$ blue points and any $n$ red points in $\mathbb R^d$, a perfect red-blue matching $\mathcal M$ that maximizes the sum of squared Euclidean distances between the matched points is a Tverberg graph. (So the sum of squared Euclidean distances of edges serves as the function $Q$ in their proof.) Next, they consider the functions $f_{ab}: \mathbb R^d\to \mathbb R$, $ab\in \mathcal M$, defined by
\[
f_{ab}(x)\coloneqq\Big(x-\frac{a+b}{2}\Big)^2-\Big(\frac{a-b}{2}\Big)^2.
\]
The resulting balancing condition in their paper takes the form
\[
o\in \conv\Big\{\frac{a+b}{2}:ab\in \mathcal M\text{ with } f_{ab}(o)=f_{\mathcal M}(o) \Big\}.
\]

In this paper, we apply this approach to prove Theorem~\ref{theorem matching}. The proof of Theorem~\ref{theorem tree} follows a different strategy, although it involves Proposition~\ref{proposition: main tool} as a key ingredient. 
\section{Intersecting balls induced by a max-sum tree}
\label{section proof of theorem tree}
\begin{proof}[Proof of Theorem~\ref{theorem tree}] 
Let $T$ be a max-sum tree of a finite point set $X\subset \mathbb R^d$.
Consider the function $H: \mathbb{R}^d \xrightarrow{}  \mathbb{R}$ defined by
\[
    H(x) \coloneqq \max\limits_{a \in X} \| x - a \|.
\]

This function attains its global minimum at a unique point, which is the center of the smallest radius ball $B$ containing $X$. Without loss of generality, we assume that this point coincides with the origin $o$. Put $r\coloneqq H(o)$, that is, $B$ is a closed ball of radius $r$. If $r=0$, then all points of $X$ coincide with $o$; hence, the theorem is trivial. Thus, we may assume that $r>0$. To complete the proof, we show that for each $pq \in E(T)$, we have
\begin{equation}
    \label{equation: origin lies in the ball B(pq)} o\in B(pq) \text{ or, equivalently, } \langle p, q\rangle\leq 0,
\end{equation}
where $\langle p, q\rangle$ stands for the dot product of $p, q \in \mathbb{R}^d$.

Let $X_r$ be a set containing the points of $X$ lying on the boundary of $B$. So, every point $a \in X \setminus X_r$ belongs to the interior of $B$. By Proposition~\ref{proposition: main tool} applied to the function $H$, we have $o\in \conv X_r$. Thus, there is a convex combination of points in $X_r$ coinciding with~$o$. 
So we can choose a non-empty set $X_r'\subset X_r$ such that there exist \textit{positive} coefficients~$\lambda_a$ for each $a\in X_r'$ satisfying
\begin{equation}
    \label{equation: origin is a weighted sum of Xr'}
    \sum \limits_{a \in X_r'} \lambda_a a = o.
\end{equation}

Let $G$ be the graph with the vertex set $V(G) = X_r'$ and two vertices $a,b\in V(G)$ are adjacent if and only if $\langle a, b \rangle \leq 0$. We claim that the graph $G$ is connected. Indeed, for any proper subset $U \subsetneq V(G)$ and its complementary set $W=V(G)\setminus U$, we have
\[
\sum \limits_{u \in U} \lambda_u u = - \sum \limits_{w \in W} \lambda_w w, \text{\ \ and thus,\ \ }
\Big\langle \sum \limits_{u \in U} \lambda_u u,  \sum \limits_{w \in W} \lambda_w w \Big\rangle  \leq 0.
\]
Since $\lambda_a>0$ for each $a\in V(G)=X_r'$, there are $u\in U$ and $w\in W$ with $\langle u, w\rangle \leq 0$, that is, the edge $uw$ connects $U$ and $W$. Therefore, the graph $G$ is connected. 

Finally, we are ready to prove~(\ref{equation: origin lies in the ball B(pq)}). Suppose to the contrary that there is an edge $p q \in E(T)$ with $\langle p, q\rangle > 0$. 
By \eqref{equation: origin is a weighted sum of Xr'}, we have
\[
\sum \limits_{a \in V(G)} \lambda_a \langle a,  p \rangle = 0,
\]
so there exists a vertex $x \in V(G)$ with $\langle x,  p \rangle \leq 0$. Similarly, there is a vertex $y \in V(G)$ with $\langle y,  q \rangle \leq 0$. The points $x$ and $y$ can coincide.

Since $G$ is a connected graph, consider a path $z_1 \dots z_k$ in $G$ connecting $z_1\coloneqq x$ and $z_k\coloneqq y$. We claim that $p q$ is the shortest line segment among $qp$, $pz_1, z_1 z_2, \dots, z_{k-1} z_k, z_kq$ (or among $qp$, $pz_1, z_1q$ if $k = 1$). Indeed, we have
\begin{align*}
\| p - z_1 \|^2 = \|p-x\|^2= \| p \|^2 + \| x \|^2 - 2 \langle p,  x \rangle &\geq \| p \|^2 + r^2 > \\
 \|p\|^2+r^2- 2 \langle p,  q \rangle &\geq  \| p \|^2 +  \| q \|^2  - 2 \langle p,  q \rangle =  \| p - q \|^2.
\end{align*}
Analogously, one can prove that $\|p-q\|$ is strictly less than $\|z_k  - q \|$ or $\|z_i  - z_{i+1} \|$ for each $i\in \{1,\dots, k-1\}$. 

Next, we will find a tree with a cost larger than that of $T$, and this contradiction will complete the proof. By deleting the edge $pq$ from $T$, we obtain the forest consisting of two trees. Let $U$ and $W$ be the vertex sets of these trees, that is, $X$ is a disjoint union of $U$ and $W$. Without loss of generality, put $p\in U$, and $q\in W$. Let $ab$ be a line segment among $pz_1,z_1z_2,\dots, z_{k-1} z_k,z_k q$ (or among $pz_1, z_1 q$ if $k = 1$) that connects $U$ and $W$. By replacing the edge $pq$ by $ab$ in the tree $T$, we obtain the desired tree.
\end{proof}

Define the \textit{cost} of a graph $G$ with respect to a function $f:\mathbb{R}_+\to \mathbb R$ as the sum
\[
    \sum_{ab\in E(G)} f(\|a-b\|),
\]
where $\mathbb R_+$ is the set of non-negative real numbers.

We say that a tree with vertex set $X$ is an \textit{$f$-max-sum tree} if it has the maximum cost among all trees with vertex set $X$.
Repeating the argument of Theorem~\ref{theorem tree}, one can easily prove the following statement. 
\begin{theorem}
Given an increasing function $f:\mathbb R_+\to \mathbb R$ and a finite point set $X$ in $\mathbb R^d$, an $f$-max-sum tree of $X$ is a Tverberg graph.
\end{theorem}

\section{Intersecting open disks induced by a max-sum matching}
\label{section proof of theorem matching}
\begin{proof}[Proof of Theorem~\ref{theorem matching}]
\label{subsection: proof of the theorem on open disks}
Let $\mathcal M$ be a max-sum matching of an even set $S$ of pairwise distinct points in the plane. Suppose, to the contrary, that the intersection of the open discs induced by $\mathcal M$ is an empty set. To prove the theorem, it is sufficient to find a matching of $S$ with the larger cost.

For each $ab\in \mathcal M$, define the function $f_{ab}:\mathbb R^2 \to \mathbb R$ by
\[
f_{ab}(x)\coloneqq\Big\| x-\frac{a+b}{2} \Big\| - \Big\| \frac{a-b}{2} \Big\|.
\]
Note that $f_{ab}(x)$ is negative if and only if $x$ belongs to the interior of the closed disc $B(ab)$. Its absolute value equals to the smallest distance between $x$ and a point on the bounding circle of $B(ab)$. Clearly, it is differentiable everywhere except at the point $(a+b)/2$. 

Also consider the function $F: \mathbb{R}^2 \xrightarrow{}  \mathbb{R}$ defined by
\[
    F(x) \coloneqq \max\limits_{a b \in \mathcal{M}} f_{ab}(x).
\]
Analogously, $F(x)$ is negative if for each $a b \in \mathcal{M}$,  the interior of $B(ab)$ contains $x$. Otherwise it equals to the radius of the smallest closed disc centered at $x$ that has a non-empty intersection with all closed discs induced by $\mathcal{M}$.

The function $F$ attains its global minimum at a unique point. Without loss of generality, we may assume that this point coincides with the origin $o$. For the sake of brevity, put $r \coloneqq F(o)$. As the open discs induced by $\mathcal M$ have no common point, the origin $o$ lies out of at least one of them, and hence, we conclude $r \geq 0$. So, the closed disc of radius $r$ centered at~$o$ is the smallest radius disc intersecting the closed disc $B(ab)$ for each $ab\in \mathcal M$. Denote the circle bounding this disc of radius $r$ by~$\Omega$.

Consider the submatching
\begin{equation}
    \label{equation: M_1 in the disc problem}
    \mathcal M_r =\Big\{ab\in \mathcal M:\ \Big\| \frac{a+b}{2} \Big\| - \Big\| \frac{a-b}{2} \Big\| = r \Big\}.
\end{equation}
Since all points of $S$ are pairwise distinct and $r \geq 0$, we have that for any edge $ab\in \mathcal M_r$, its midpoint $\frac{a+b}{2}$ is distinct from the origin $o$. Thus for $ab\in \mathcal M_r$, the function $f_{ab}$ is differentiable at the origin. By Proposition~\ref{proposition: main tool} applied to the function $F$, we obtain
$$o\in \conv \Big\{\nabla f_{ab}(o): ab\in \mathcal M_r \Big\} \text{,\ \ and thus,\ \ }  o \in \conv \Big\{\frac{a+b}{2} : ab \in \mathcal M_r\Big\}.$$

Recall that the midpoint of any edge in $\mathcal M_r$ does not coincide with $o$. Thus, the origin either lies on a line segment connecting two midpoints or belongs to the interior of a triangle defined by three midpoints. (Formally, this follows from the two-dimensional case of Carath\'eodory's theorem.) So we denote by $\mathcal M_r'$ any submatching of $\mathcal M_r$ of size $2$ or $3$ that satisfies
\begin{equation}
    \label{equation: origin lies in the convex hull of centers}
    o\in\conv \Big\{ \frac{a+b}{2}:ab\in \mathcal M_r'\Big\}.
\end{equation}

\begin{figure}[h!]
	\centering
	\includegraphics{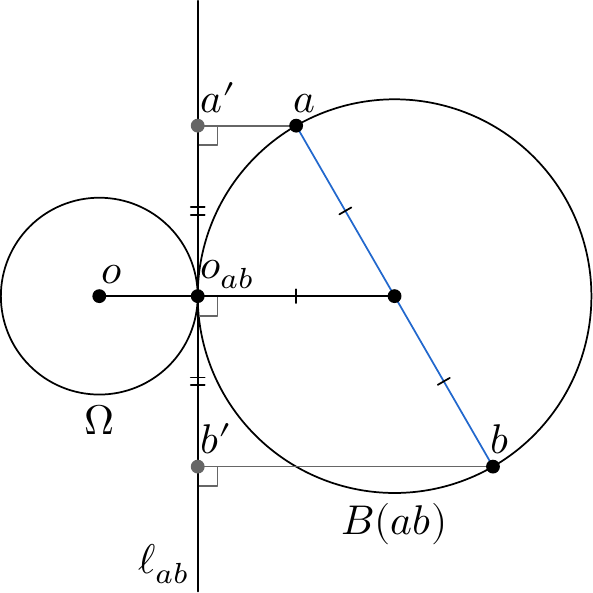}
	\caption{Notation used in the proof of Theorem~\ref{theorem matching}}
	\label{fig:im1}
\end{figure}
Recall that $\Omega$ is the circle centered at the origin of radius $r$. By \eqref{equation: M_1 in the disc problem}, we have that $\Omega$ externally touches the discs $B(a b)$ for $a b \in \mathcal M_r' \subseteq \mathcal M_r$. Let $o_{ab}$ be the only common point of~$\Omega$ and $B(ab)$. Let $\ell_{ab}$ be the common internal tangent line to $\Omega$ and $B(a b)$ passing through $o_{ab}$; see Figure~\ref{fig:im1}. Let $a'$ and $b'$ be the orthogonal projections of the points $a$ and $b$ onto the line~$\ell_{ab}$. Put $r_{a}\coloneqq\|a-a'\|$ and $r_{b}\coloneqq\|b-b'\|$. Denote by $B_a$ and $B_b$ the closed discs of radius $r_a$ and $r_b$ centered at $a$ and $b$, respectively. It is possible that one of them is of radius 0. Hence the line $\ell_{ab}$ touches the discs $B_a$, $B(ab)$, and $B_b$; see Figure~\ref{fig:im2}. 

Since the point $o_{ab}$ is the projection of $\frac{a+b}{2}$ onto the line $\ell_{ab}$ and the point $\frac{a+b}{2}$ is the midpoint of the line segment $ab$, the point $o_{ab}$ is the midpoint of the segment $a'b'$. Moreover,
\begin{equation}
    \label{equation: touching circles}
    r_{a} + r_{b} = 2\Big\| \frac{a+b}{2}-o_{ab}\Big\|=2\Big(\Big\| \frac{a + b}{2} \Big\|-r\Big)=2\|a-b\|.
\end{equation}
Here the second equality follows from the fact that the points $o$, $o_{ab}$, and $\frac{a+b}{2}$ are collinear.
This fact combined with \eqref{equation: origin lies in the convex hull of centers} also implies $o\in \conv \{o_{ab}:ab\in \mathcal M_r'\}.$ Thus, 
\begin{equation}
    \label{equation: convex hull of o_i}
\text{if $r>0$, then any closed semicircle of $\Omega$ contains a point $o_{ab}$, where $ab\in \mathcal M_r'$.}
\end{equation}

Denote by $S_r\subseteq S$ the set of all endpoints of segments of $\mathcal M_r'$. Let $G=(V(G), E(G))$ be the graph with $V(G)=S_r$, whose edge set $E(G)$ is the union of the sets
\[
    E_b(G) =\mathcal M_r' \;\text{ and }\; E_r(G) =\{ab :a,b\in V(G) \text{ with } \| a - b \| > r_a + r_b\}.
\]
We call edges in $E_b(G)$ and $E_r(G)$ \textit{blue} and \textit{red}, respectively. 
For a blue edge $ab$, we have~\eqref{equation: touching circles}, that is, the discs $B_a$ and $B_b$ touch each other externally. Also, $cd$ is a red edge if and only if the discs $B_c$ and $B_d$ are disjoint. 

Let us show that if $G$ contains an alternating cycle, then there is a matching with the larger cost than that of $\mathcal M$. Here, an \textit{alternating cycle} is a simple cycle of even length whose edges are taken alternately from $E_b(G)$ and $E_r(G)$. Indeed, assume that there is a cycle $x_1 y_1 \dots x_m y_m x_{m+1}$, where $x_{m+1} \coloneqq x_1$, such that $x_i y_i \in E_b(G)$ and $y_ix_{i+1}\in E_r(G)$. Hence we have
$$\sum \limits_{i=1}^{m} \lVert y_i - x_{i+1} \rVert > \sum \limits_{i=1}^{2m} (r_{y_i} + r_{x_{i+1}}) = \sum \limits_{i=1}^{2m} (r_{y_i} + r_{x_i}) = \sum \limits_{i=1}^{m} \lVert x_i - y_i \rVert.$$
Replacing in \(\mathcal M\) the blue edges from the alternating cycle with the red ones, we obtain the desired perfect matching
\[
\mathcal M\setminus \{x_1y_1,\dots,x_my_m\}\cup \{ y_1x_2,\dots, y_mx_{m+1}\}.
\]
Thus, finding an alternating cycle in $G$ is sufficient to finish the proof. For that, we use the following two lemmas shown in the next subsection.
\begin{lemma}\label{circle_lemma_1}
If $r > 0$, then for any blue edge $a b$, there is another blue edge $cd$ such that either $ac, ad \in E_r(G)$ or $bc, bd\in E_r(G)$.
\end{lemma}
\begin{lemma}\label{circle_lemma_2}
If $r = 0$, then either the graph $G$ contains an alternating cycle of length 4 or for any blue edge $ab$, there is another blue edge $cd$ such that either $ac, bd\in E_r(G)$ or $bc, bd\in E_r(G)$.
\end{lemma}

Suppose to the contrary that $G$ contains no alternating cycle. By Lemmas~\ref{circle_lemma_1} and~\ref{circle_lemma_2} applied to the blue edge $a_1b_1$, without loss of generality, we may assume that there is another blue edge $a_2b_2$ such that $a_1a_2$ and $a_1b_2$ are red edges. Since there is no alternating cycle in $G$, there are no more red edges among vertices $a_1, b_1, a_2, b_2$. Thus, by Lemmas~\ref{circle_lemma_1} and~\ref{circle_lemma_2} applied to the blue edge $a_2b_2$, we may assume that there is a third blue edge $a_3b_3$ such that $a_2a_3$ and $a_2 b_3$ are red edges. (In particular, the size of $E_b(G)=\mathcal M_r'$ is 3.) Applying a similar argument, without loss of generality, we may assume that the edges $a_3 a_1$ and $a_3 b_1 $ are red. Hence the graph $G$ contains the alternating cycle $a_1 b_2 a_2 b_3 a_3 b_1$, a contradiction. This finishes the proof of Theorem~\ref{theorem matching}.
\end{proof}

\subsection{Proof of auxiliary lemmas}

For a blue edge $ab\in \mathcal M_r'$, let $u_{ab}$ be the tangency point of the discs $B_a$ and $B_b$. Denote by $s_{ab}$ the common internal tangent of $B_a$ and $B_b$; see Figure~\ref{fig:im2}. Then the line $s_{ab}$ is perpendicular to the line segment $ab$ and passes through the point $u_{ab}$.

To prove Lemmas~\ref{circle_lemma_1} and~\ref{circle_lemma_2}, we will use the following simple observation from elementary geometry: \textit{The line $s_{ab}$ passes through the point $o_{ab}$}. (Recall that in the case $r=0$, the point $o_{ab}$ and the origin $o$ coincide.) 
Indeed, if one of the discs $B_a$ or $B_b$ is a singleton, then it coincides with the point $o_{ab}$, and the tangent line $s_{ab}$ passes through this point (in this case $s_{ab} = \ell_{ab})$. Otherwise the common external tangent $\ell_{ab}$ of $B_a$ and $B_b$ intersects $s_{ab}$ at a unique point $v_{ab}$. Recall that two tangent segments drawn from an external point to a disc have the same length. This property implies that 
\[
\|v_{ab} - a'\| = \|v_{ab} - u_{ab}\| = \|v_{ab} - b'\|.
\] 
Since $o_{ab}$ is the midpoint of the segment $a'b'$, it must coincide with the point $v_{ab}$. Hence $o_{ab}$ lies on $s_{ab}$. 
%Indeed, the line $s_{ab}$ is the radical axis
%\footnote[1]{Denoting by $\sigma(u, r)$ a circle of radius $r$ centered at a point $u$, recall standard definitions from elementary plane geometry; for example, see the Wikipedia pages~\cite{} and~\cite{}. The \textit{power} of a point $v$ with respect to the circle $\sigma(u,r)$ is the real number $\|u-v\|^2-r^2$. This concept has a simple geometric interpretation if $v$ does not belong to the interior of the disc bounded by $\sigma(u,r)$. Given a tangent line $vw$ to $\sigma(u,r)$, where $w\in \sigma(u,r)$, the Pythagorean theorem yields that the power of $v$ with respect to $\sigma(u,r)$ equals $\|v-w\|^2$. The \textit{radical axis} of two circles is the set of points have the same powers with respect to these circles. Particularly, if a point belongs to the radical axis of two circles and lies out of the circles, then its tangents to the circles are of the same length. It is well-known that if two circles are not concentric, then their radical axis is a straight line orthogonal to the line connecting their centers.}
%of the circles bounding the discs $B_a$ and $B_b$. As $o_{ab}$ is the midpoint of the segment $a'b'$, the powers of $o_{ab}$ with respect to those circles are equal. Hence $o_{ab}$ lies on $s_{ab}$.

\begin{proof}[Proof of Lemma \ref{circle_lemma_1}]
First, we show that if $a$ coincides with $o_{ab}$, then it is connected with all other vertices but $b$ by red edges. For any edge $cd \in \mathcal M_r'$ distinct from $ab$, the tangent line $\ell_{cd}$ partitions the plane into the open half-plane containing $\Omega \setminus \{o_{c d}\}$ and the closed half-plane containing $B_c$ and $B_d$. Hence $B_a = \{o_{a b}\} \subset \Omega \setminus \{o_{cd}\}$ lies in the open half-plane. (Here we use that if  $cd\ne a b$, then the points $o_{ab}\ne o_{cd}$, which follows from the minimality of $\mathcal M_r'$.) Therefore, the discs $B_c$ and $B_a$ are strictly separated, and thus, $a c$ is a red edge. Similarly, the edge $a d$ is also red.

From now on, we may assume that $a$ and $b$ are distinct from $o_{ab}$. Denote by $\omega(a)$ the subset of $\Omega$ satisfying the following property. A point $t \in \Omega$ belongs to $\omega(a)$ if and only if $B_a$ lies in the same open half-plane as the origin $o$ with respect to the tangent line to $\Omega$ passing through~$t$. The set $\omega(a)$ is an open arc whose endpoints correspond to the external tangent lines of $\Omega$ and $B_a$. Analogously, we define $\omega(b)$. Clearly, $o_{ab}\not \in \omega(a)$ and $o_{ab}\not\in \omega(b)$.
Since the tangent line to $\Omega$ at $-o_{a b}$ is parallel to $\ell_{a b}$, we conclude that $-o_{a b}$ belongs to the sets $\omega(a)$ and $\omega(b)$, simultaneously. Therefore, the union $\omega(a) \cup \omega(b)$ is also an open arc.

%\begin{figure}[h!]
%	\centering
%   \includegraphics{images/rpositive.pdf}
%	\caption{Case $r > 0$}
%	\label{fig:im21}
%\end{figure}

Since the points $a$ and $b$ are distinct from $o_{ab}$, the line $s_{a b}$ meets the open line segment $ab$. Thus, the lines $s_{ab}$ and $\ell_{a b}$ are distinct. Hence $s_{a b}$ intersects the circle $\Omega$ at two points, one of them is $o_{ab}$. Consider two distinct tangent lines $s_a$ and $s_b$ to $\Omega$ that are parallel to $s_{a b}$; see Figure~\ref{fig:im21}. Clearly, the line $s_{ab}$ lies between $s_a$ and $s_b$. Without loss of generality, we may assume that for $x\in \{a,b\}$, the origin $o$ and the disc $B_x$ lie on the same side with respect to $s_x$. Since the lines $s_a$ and $s_b$ touch $\Omega$ at antipodal points, the open arc $\omega(a) \cup \omega(b)$ contains these points. Thus, it also contains a closed semicircle of $\Omega$. By (\ref{equation: convex hull of o_i}), the open arc $\omega(a) \cup \omega(b)$ contains some point $o_{cd}$ for $cd\in \mathcal M_r'$.

Without loss of generality, we may assume $o_{cd} \in \omega(a)$. Hence, the line $\ell_{cd}$ partitions the plane into two half-planes: One of them is open and contains the origin and the disc $B_a$ and another is closed and contains the discs $B_c$ and $B_d$. Therefore, we conclude that $ac$ and $ad$ are red edges.
\end{proof}
\begin{figure}[htp]
    \centering
    \begin{subfigure}[b]{0.5\textwidth}
        \centering
        \includegraphics{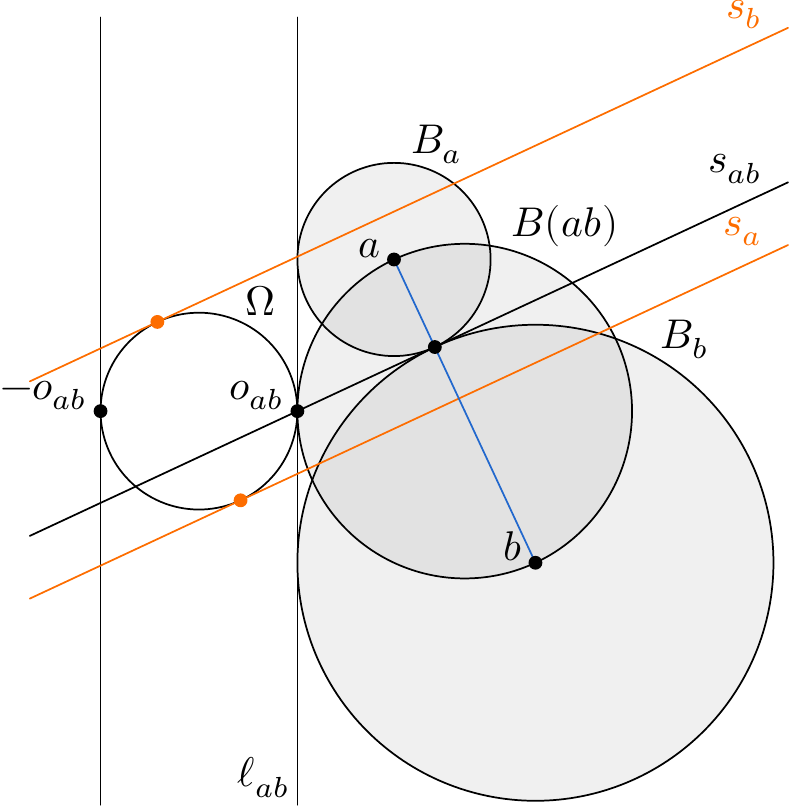}
    \caption{$r > 0$}
    \label{fig:im21}
    \end{subfigure}
    \begin{subfigure}[b]{0.45\textwidth}
        \centering
        \includegraphics{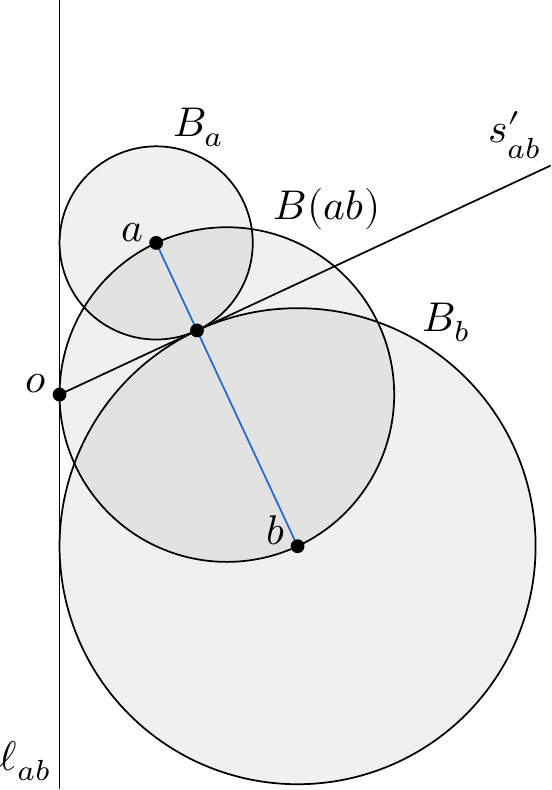}
        \caption{$r = 0$}
    \label{fig:im22}
    \end{subfigure}
	\caption{Proofs of Lemmas~\ref{circle_lemma_1} and \ref{circle_lemma_2}}
    \label{fig:im2}
\end{figure}
\begin{proof}[Proof of Lemma \ref{circle_lemma_2}]

The proof of this lemma is similar to the argument of Lemma~\ref{circle_lemma_1}.

For each edge $a b \in \mathcal M_r'$, the point $o_{a b}$ coincides with the origin $o$. Thus, $o$ lies on the lines $\ell_{a b}$ and $s_{a b}$ and belongs to the boundary of $B(ab)$; see Figure \ref{fig:im22}. Since all points of $S$ are distinct, at most one vertex of $G$ coincides with the origin.

First, we prove that if some two discs $B(ab)$ and $B(cd)$ touch externally, then there is an alternating cycle of length 4. Indeed, then the lines $\ell_{ab}$ and $\ell_{cd}$ coincide and the sets $\{B_a, B_b\}$ and $\{B_c,B_d\}$ lie in the opposite closed half-planes induced by the line~$\ell_{ab}=\ell_{cd}$. Suppose that at least one of the pairs $ac, ad, bc, bd$ is not a red edge. (Otherwise, we easily find an alternating 4-cycle.)
Without loss of generality, assume that the pair $ac$ is \textit{not} a red edge. Then the discs $B_a$ and $B_c$ touch each other at the point $a'=c'$ lying on the line $\ell_{ab}=\ell_{cd}$. We know that the point $b'$ is symmetric to $a'$ with respect to $o_{ab}=o$. Similarly, $d'=-c'$ and therefore $b'=d'$, which implies that the discs $B_b$ and $B_d$ touch each other externally. Note that the point $a'=c'$ coincides with the point $b'=d'$ if and only if the points $a',b',c',d'$ coincide with the origin $o$, that is, two of the points $a,b,c,$ and $d$ coincide with the origin, which is impossible. So, the points $a'=c'$ and $b'=d'$ are distinct. Hence the discs $D_a$ and $D_c$ are disjoint, that is, $ac$ is a red edge. Analogously, $bd$ is a red edge, and so, the 4-cycle $abdc$ is alternating. 

From now on, we assume that among the discs induced by $\mathcal M_r'$ there are no two touching externally. Thus, by~\eqref{equation: origin lies in the convex hull of centers} and the minimality of $\mathcal M_r'$, no two vectors of ${X \coloneqq \{\frac{c+d}{2}:cd\in \mathcal M_r'\}}$ are collinear.

Next, we show that for any blue edge $a b$, there is a blue edge $cd$ such that either $ac, ad\in E_r(G)$ or $b c, bd\in E_r(G)$. First, we consider the case when the point $a$ coincides with the origin $o$. We claim that the vertex $a$ is connected with all other vertices of $G$ but $b$ by red edges. Indeed, for any edge $cd\in\mathcal M_r'$ distinct from $ab$, the disc $B_c$ touches the line $\ell_{cd}$ passing through $o_{cd}=o=a$ at a unique point $c'$. If the point $c'$ coincides with the origin, so does $c$ or $d$. It is impossible because all points of $S$ are distinct. Hence the disc $B_c$ does not intersect $B_a = \{o\}$, and so, the edge $ac$ is red. From now on, we may assume that the vertices $a$ and $b$ are distinct from the origin.

Let $s'_{a b}\subset s_{ab}$ be the ray emanating from $o$ and passing through the only common point of $B_a$ and $B_b$. For any $cd\in \mathcal M_r'$, denote by $H_{cd}^+$ the closed half-plane bounded by the line $\ell_{cd}$ and containing $B_c$ and $B_d$. Put $H_{cd}^- \coloneqq \mathbb R^2\setminus H_{cd}^+$. Suppose that for any $cd \in \mathcal M_r'$ the ray $s'_{ab}$ lies in $H_{cd}^+$. Then the angle between the ray $s'_{ab}$ and each vector of $X$ is either right or acute. By \eqref{equation: origin lies in the convex hull of centers}, we have $o\in \conv X$, which is possible only if the set $X$ has a pair of collinear vectors in the opposite directions, a contradiction. Hence, there exists an edge $cd \in \mathcal M_r'$ such that $s'_{ab}\setminus \{o\} \subset H_{cd}^-$. As the lines $\ell_{ab}$ and $\ell_{cd}$ are distinct, the intersection of the open half-plane $H_{cd}^-$ and the line $\ell_{ab}$ is an open ray. Therefore, $H_{cd}^-$ contains one of the discs $B_a$ or $B_b$ touching this open ray and $s'_{ab}\setminus\{o\}\subset H_{cd}^-$. Hence, the line $\ell_{cd}$ separates the discs $B_c, B_d\subset H_{cd}^+$ from one of the discs $B_a$ or $B_b$ lying in $H_{cd}^-$. So, either $ac, ad\in E_r(G)$ or $b c, bd\in E_r(G)$, which finishes the proof of Lemma~\ref{circle_lemma_2}.
\end{proof}

\section{Discussion}
\label{section discussion}

In this section, we explore the connection between Theorem~\ref{theorem matching} and the following result of Bereg et al.~\cite{bereg2023maximum}.
\begin{theorem}[\cite{bereg2023maximum}*{Theorem~3.14}]
\label{theorem bereg}
A max-sum matching of any even set of points in the plane is a Tverberg graph.
\end{theorem}

At first, we use Theorem~\ref{theorem matching} to derive Theorem~\ref{theorem bereg}. Subsequently, we analyze whether the approach of Bereg et al.~\cite{bereg2023maximum} works to ensure the common intersection of open disks if all points in a set are distinct.

\subsection{Theorem~\ref{theorem matching} implies Theorem~\ref{theorem bereg}}

The idea is to properly approximate a given point set by sets with pairwise distinct points and apply Theorem~\ref{theorem matching}.
The key observation is that elongating an edge of a max-sum matching yields another max-sum matching, which is formally stated in the following lemma proposed by Bereg et al.~\cite{bereg2023maximum}. We provide a proof for the sake of completeness. 

\begin{lemma}[\cite{bereg2023maximum}*{Lemma~3.5}]
\label{extension lemma}
Let $\mathcal{M}$ be a max-sum matching of an even set of points $S \subset \mathbb R^2$, and let $ab \in \mathcal{M}$. For any point $c \in \mathbb R^2$ such that $b$ lies on the line segment $ac$, the matching $(\mathcal{M} \cup \{ac\}) \setminus \{ab\}$ has the maximum cost among all perfect matchings with the vertex set $(S \cup \{c\}) \setminus \{b\}$.
\end{lemma}
\begin{proof}
Consider an arbitrary perfect matching $\mathcal M'$ with the vertex set $(S \cup \{c\}) \setminus \{b\}$. Let $d \in S$ be a vertex adjacent to $c$ in $\mathcal M'$. Then $(\mathcal M' \cup \{bd\}) \setminus \{cd\}$ is a perfect matching the with vertex set $S$. Using the triangle inequality and the maximality of $\mathcal{M}$, we obtain
\begin{align*}
    \cost \mathcal M' = &\cost ((\mathcal M' \cup \{bd\}) \setminus \{cd\}) + \|c - d\| - \|b - d\| \leq \cost \mathcal M + \|c - b\| = \\ &\cost \mathcal M + \|a - c\| - \|a - b\| = \cost \mathcal ((\mathcal M \cup \{ac\}) \setminus \{ab\}),
\end{align*}
concluding the proof.
\end{proof}

\begin{proof}[Theorem~\ref{theorem matching} $\implies$ Theorem~\ref{theorem bereg}]
Let $\mathcal{M}$ be a max-sum matching of an even set of points $S$. For every edge $ab \in \mathcal{M}$, consider two sequences of points $\{a_n\}_{n = 1}^{\infty}$ and $\{b_n\}_{n = 1}^{\infty}$ converging to $a$ and $b$, respectively, such that $\{a_n b_n\}_{n = 1}^{\infty}$ is a nested sequence of line segments intersecting at the line segment $ab$.
Next, define a sequence of matchings $\{\mathcal M_n\}_{n = 1}^{\infty}$ with 
%For each $n$ define a new matching $\mathcal M_n$ as
$\mathcal M_n \coloneqq \{a_n b_n : ab \in \mathcal M \}$. Denote by $S_n$ the vertex set of $\mathcal M_n$. Since $S$ is finite, we can additionally ensure that all points of $S_n$ are pairwise distinct. 

Applying Lemma~\ref{extension lemma} multiple times, we deduce that $\mathcal M_n$ is a max-sum matching of the point set $S_n$. By Theorem~\ref{theorem matching}, open discs induced by $\mathcal M_n$ share at least one common point that we denote by $x_n$. Note that $x_n \in B(a_m b_m)$ whenever $n \geq m$ for all $ab \in \mathcal{M}$, and thus, the sequence $\{x_n\}_{n = 1}^{\infty}$ is bounded. By compactness, its subsequence converges to a point $x \in \mathbb R^2$. For each $ab \in \mathcal{M}$, the point $x$ lies in the intersection $\cap_{n = 1}^{\infty} B(a_n b_n) = B(a b)$. Hence we have
\[
    x \in \bigcap_{ab\in \mathcal{M}} B(ab),
\]
which finishes the proof.
\end{proof}

\subsection{Does Theorem~\ref{theorem bereg} imply Theorem~\ref{theorem matching}?}

Note that Theorem~\ref{theorem bereg} immediately implies Theorem~\ref{theorem matching} for an even set $S$ with a \textit{unique} max-sum matching $\mathcal M$. Indeed, by slightly moving each point of $S$ towards its matched point in $\mathcal M$, we obtain a new point set such that its matching $\mathcal M'$ corresponding to $\mathcal M$ is also max-sum. By Theorem~\ref{theorem bereg}, the matching $\mathcal M'$ is a Tverberg graph. Since the closed disc induced by any edge in the new matching $\mathcal M'$ lies in the interior of the closed disc induced by the corresponding edge in $\mathcal M$, the matching $\mathcal M$ is an open Tverberg graph. 
Unfortunately, we did not find a concise way to deduce Theorem~\ref{theorem matching} for sets with a few max-sum matchings from Theorem~\ref{theorem bereg}. However, we believe that the approach of Bereg et al.~\cite{bereg2023maximum} can be adapted to prove Theorem~\ref{theorem matching}. Next, we briefly expose their argument and point out the issue, which formally requires an additional explanation to complete the proof of Theorem~\ref{theorem matching}.

\begin{figure}[h!]
	\centering
  \includegraphics{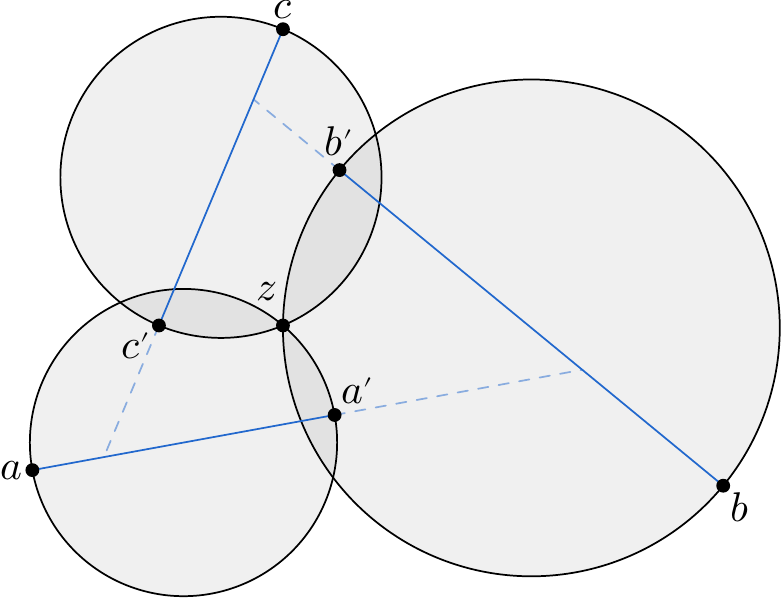}
	\caption{Since the discs induced by the matching $\{aa',bb',cc'\}$ share a single point $z$, then the matching is not max-sum.}
	\label{figure bereg proof}
\end{figure}

%In their proof of Theorem~\ref{theorem bereg}, Bereg et al.~\cite{bereg2023maximum} use Helly’s theorem to reduce the general problem to sets of 6 points. 
First, Bereg et al.~\cite{bereg2023maximum} reduce Theorem~\ref{theorem bereg} to the case when a max-sum matching consists of 3 edges (that is, there are exactly 6 point in a set). Then they observe that there are only 10 possible arrangements of these 3 edges, which they enumerate from (A) to (J); for details, we refer to Figure~4 in~\cite{bereg2023maximum}. They study each configuration independently and show that in each case the closed discs induced by the matching intersect. Moreover, in the cases from (A) to (G), their arguments ensure that the corresponding open discs also intersect under the assumption that the points are distinct. In the remaining cases, they conclude that it is enough to show the following fact: \textit{If a matching $\mathcal M$ consists of 3 edges arranged according to one of the configurations (H), (I), or (G) and the closed discs induced by $\mathcal M$ share a single point distinct from the vertices of $\mathcal M$, then $\mathcal M$ is not a max-sum matching}. Figure~\ref{figure bereg proof} illustrates what happens in the case (H). Finally, they confirm this fact, thereby completing the proof of Theorem~\ref{theorem bereg}. 

Remark that their argument of the fact relies on the assumption that the common point of the discs is distinct from the vertices. Unfortunately, to formally show Theorem~\ref{theorem matching} using the proof from~\cite{bereg2023maximum}, one needs to confirm the fact without this assumption. Essentially, this is the only step missing to complete the proof of Theorem~\ref{theorem matching}. However, given this theorem holds, we believe that this can be verified without involving advanced machinery.

\section{Open problems}
\label{section: open problems}

While a max-sum tree of a finite set of distinct points in $\mathbb R^d$ is a Tverberg graph, it may not necessarily be an open Tverberg graph. For instance, consider a rhombus where one diagonal is shorter than its side. In this case, any max-sum tree of the points in the rhombus consists of the largest diagonal and two opposite sides, as illustrated in Figure~\ref{figure rhombus}. It is easy to verify that the intersection of open discs induced by this max-sum tree is empty. 

Pirahmad et al.~\cite{pirahmad2022intersecting}*{Theorem 1.3} established that for any even set of distinct points in $\mathbb{R}^d$, there exists a perfect matching that is an open Tverberg graph. They also questioned whether a max-sum matching satisfies this property and posed the following problem.

\begin{figure}[h!]
	\centering
  \includegraphics{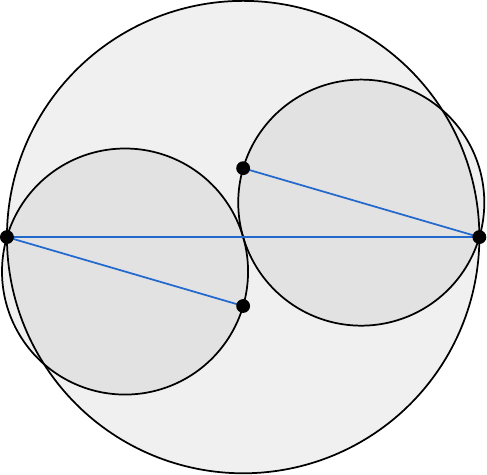}
	\caption{Max-sum tree that is not an open Tverberg graph}
	\label{figure rhombus}
\end{figure}

\begin{problem}[\cite{pirahmad2022intersecting}*{Problem~9.3}]
\label{problem max-sum matching}
Is it true that for any even set of distinct points in $\mathbb R^d$, a max-sum matching is an open Tverberg graph?
\end{problem}

Theorem~\ref{theorem matching} answers this question affirmatively for $d=2$ and some of our observations work for any $d > 2$. In particular, we can consider an analogous function $F$ for $d > 2$ and apply Proposition~\ref{proposition: main tool} to reduce the problem to finding an alternating cycle in a similarly defined graph. Nevertheless, our approach heavily relies on the arrangement of points and cannot be directly extended to higher dimensions. Therefore, this problem remains open in general and it is an intriguing area for further research.

%Although we believe that the answer to this problem is positive, we are not able to confirm it in its generality. 
Moreover, we conjecture a quantitative variation of this problem, which resembles the quantitative results and problems of B\'ar\'any, Katchalski, and Pach; see~\cite{barany1982quantitative}.

\begin{conjecture}\label{conjecture quantitative}
For a positive integer $d$, there exists a constant $\varepsilon_d>0$ such that any even set $X$ of (distinct) points in $\mathbb R^d$ with minimum distance $r>0$ satisfies the following property. The intersection of the closed balls induced by a max-sum matching $\mathcal M$ of $X$
contains a ball of radius $\varepsilon_d r$.
\end{conjecture}

It seems that our approach developed to prove Theorem~\ref{theorem matching} does not allow us to confirm Conjecture~\ref{conjecture quantitative} even in the plane. So, it remains interesting to verify it in this special case.

\subsection*{Declaration of competing interest}

The authors declare that they have no known competing financial interests or personal relationships that could have appeared to influence the work reported in this paper.

\subsection*{Data availability}

No data was used for the research described in the article.

\subsection*{Acknowledgements}

The research of P.B. was funded by the grant of Russian Science Foundation No. 21-71-10092, \url{https://rscf.ru/project/21-71-10092/}. A.P. is supported by the Bulgarian Ministry of Education and Science, Scientific Programme "Enhancing the Research Capacity in Mathematical Sciences (PIKOM)", No. DO1-67/05.05.2022, and the Young Russian Mathematics award.
\bibliographystyle{siam}
\bibliography{biblio}

\end{document}